\documentclass[a4paper,11pt]{article}
\pdfoutput=1
\usepackage{hyperref}
\hypersetup{hypertexnames = false, bookmarksdepth = 2, bookmarksopen = true, colorlinks, linkcolor = black, citecolor = black, urlcolor = black, pdfstartview={XYZ null null 1}}

\usepackage{amsfonts}
\usepackage[fleqn, leqno]{amsmath}
\usepackage{amsthm}

\usepackage{zref-clever}
\zcsetup{cap}
\zcsetup{noabbrev}
\newcommand{\cref}[1]{\zcref{#1}}
\newcommand{\Cref}[1]{\zcref[S]{#1}}

\usepackage{booktabs}
\usepackage{colonequals}
\usepackage{enumitem}
\usepackage[vmargin=1in, hmargin=0.9in, includefoot, headheight=0pt, headsep=0pt, footskip=30pt]{geometry}
\usepackage{mathtools}
\usepackage{parskip}
\usepackage{thmtools}
\usepackage{tikz-cd}
\usepackage{xparse}
\usepackage{xpatch}
\usepackage{xspace}

\usepackage[utf8]{inputenc}
\usepackage[T1]{fontenc}
\usepackage{libertine}
\usepackage[libertine]{newtxmath}
\usepackage[scaled=0.83]{beramono}
\usepackage{eucal}
\usepackage{microtype}
\usepackage{gitinfo2}

\usepackage[backend=biber, maxbibnames=99, datamodel=mrnumber, sortcites]{biblatex}

\DeclareDelimFormat{postnotedelim}{\addcomma\addnbspace}

\DeclareFieldFormat{mrnumber}{%
  MR\addcolon\space
  \ifhyperref
  {\href{http://www.ams.org/mathscinet-getitem?mr=#1}{\nolinkurl{#1}}}
{\nolinkurl{#1}}}

\renewbibmacro*{doi+eprint+url}{%
  \iftoggle{bbx:doi}
  {\printfield{doi}}
  {}%
  \newunit\newblock
  \printfield{mrnumber}%
  \newunit\newblock
  \iftoggle{bbx:eprint}
  {\usebibmacro{eprint}}
  {}%
  \newunit\newblock
  \iftoggle{bbx:url}
  {\usebibmacro{url+urldate}}
{}}

\DeclareSourcemap{
  \maps[datatype=bibtex]{
    \map{
      \step[fieldset=issn, null]
    }
  }
}
\xpretobibmacro{bib+doi+url}
{\iffieldundef{doi}
  {}
{\clearfield{url}}}
{}{}

\xpretobibmacro{cite+doi+url}
{\iffieldundef{doi}
  {}
{\clearfield{url}}}
{}{}

\makeatletter
\g@addto@macro\thmt@autorefsetup{%
  \@xa\def\csname\thmt@envname autorefname\@xa\endcsname\@xa{\thmt@thmname}%
}
\patchcmd\@thm
  {\@topsepadd \thm@postskip}
  {\advance\thm@postskip-\parskip \@topsepadd \thm@postskip}
  {}{\typeout{Couldn't patch \string\@thm for parskip}}
\AddToHook{env/proof/end}{\advance\@topsepadd-\parskip}
\makeatother
\declaretheoremstyle[spaceabove = 3pt, spacebelow = 3pt, bodyfont = \itshape]{theorem}
\declaretheoremstyle[spaceabove = 3pt, spacebelow = 3pt]{remark}

\declaretheorem[style=theorem]{theorem}
\declaretheorem[style=theorem, sibling=theorem]{corollary}
\declaretheorem[style=theorem, sibling=theorem]{lemma}
\declaretheorem[style=theorem, sibling=theorem]{proposition}

\declaretheorem[style=remark, sibling=theorem]{definition}

\declaretheorem[style=remark, sibling=theorem]{remark}

\declaretheorem[style=theorem, numberwithin=section, title=Theorem]{alphatheorem}

\zcRefTypeSetup{alphatheorem}{Name-sg=Theorem,name-sg=theorem,Name-pl=Theorems,name-pl=theorems}
\zcRefTypeSetup{alphaconjecture}{Name-sg=Conjecture,name-sg=conjecture,Name-pl=Conjectures,name-pl=conjectures}
\zcRefTypeSetup{alphacorollary}{Name-sg=Corollary,name-sg=corollary,Name-pl=Corollaries,name-pl=corollaries}
\zcRefTypeSetup{alphaproposition}{Name-sg=Proposition,name-sg=proposition,Name-pl=Propositions,name-pl=propositions}

\zcRefTypeSetup{enumi}{refbounds={,\textup{(},\textup{)},}}

\setlist[enumerate]{font=\normalfont}

\mathchardef\mhyphen="2D

\usepackage[margin=2em]{caption}
\usetikzlibrary{calc}

\DeclareMathOperator\Br{Br}
\DeclareMathOperator{\Hom}{Hom}
\DeclareMathOperator{\Ext}{Ext}
\DeclareMathOperator{\ext}{ext}
\DeclareMathOperator{\codim}{codim}

\DeclareMathOperator\Pic{Pic}
\DeclareMathOperator{\dimvect}{\underline{\operatorname{dim}}}

\DeclareDocumentCommand\modulispace{om}{\IfNoValueTF{#1}{\mathrm{M}{(#2)}}{\mathrm{M}^{#1}(#2)}}
\DeclareDocumentCommand\repspace{om}{\IfNoValueTF{#1}{\mathrm{R}{(#2)}}{\mathrm{R}^{#1}(#2)}}

\newcommand{\field}{\mathbf{k}}                           %
\newcommand\tuple[1]{\ensuremath{\boldsymbol #1}}

\newcommand{\stable}{\ensuremath{\mhyphen\mathrm{st}}}
\newcommand\gen{\ensuremath{\mathrm{gen}}}

\newcommand\stratum{\ensuremath{\mathrm{S}}}

\usepackage{xpatch}

\makeatletter
\patchcmd\blx@bblinput{\blx@blxinit}
                      {\blx@blxinit
                      }{}{\fail}
\makeatother

\AtBeginBibliography{\small}

\title{Ample stability implies rigidity for quiver moduli}
\author{Pieter Belmans \and Gianni Petrella}

\begin{document}
\maketitle

\begin{abstract}
  We show that ample stability suffices to apply Teleman quantisation
  to the endomorphism bundle of the universal representation on a quiver moduli space,
  and thus in particular to deduce its infinitesimal rigidity.
  This answers a question left open in our earlier work
  with Brecan, Franzen, and Reineke on the sheaf cohomology of universal bundles.
  This sufficiency makes several recent results applicable in the widest possible range,
  subject only to ample stability,
  which for coprime dimension vectors guarantees that the Picard rank of the moduli space is maximal.
\end{abstract}

\section{Introduction}
\label{section:introduction}
Quiver moduli are algebraic varieties
attached to the very concrete situation of quivers and their representations,
whose geometry can be studied with powerful combinatorial tools.

This note is concerned with comparing ample stability
to an inequality which arose in \cite{MR4954467},
and which is sandwiched between ample and strong ample stability.
The inequality we are interested in, \cite[Equation~(34)]{MR4954467}\footnote{%
  We correct two typos in \cite[Equation~(34)]{MR4954467}:
  the left-hand side should read~$\mu(\tuple{d}^1)-\mu(\tuple{d}^\ell)$,
  and the arguments of the Euler form should read~$\langle\tuple{d}^m,\tuple{d}^n\rangle$,
in accordance with \cite[Corollary~3.20]{MR4954467}.},
which we will refer to as the \emph{rigidity inequality},
is attached to a non-trivial Harder--Narasimhan type~$\tuple{d}^*$
and reads
\begin{equation}
  \label{equation:rigidity-inequality}
  \mu(\tuple{d}^1)-\mu(\tuple{d}^\ell)
  <
  \sum_{1\leq m<n\leq\ell}\left( \mu(\tuple{d}^m)-\mu(\tuple{d}^n) \right)\left( -\langle\tuple{d}^m,\tuple{d}^n\rangle \right).
\end{equation}
Here~$\mu$ denotes the slope,
$\langle-,-\rangle$ is the Euler form,
and~$\tuple{d}^1,\ldots,\tuple{d}^\ell$ are the~$\mu$-semistable parts of
the Harder--Narasimhan type~$\tuple{d}^*$,
listed in order of strictly decreasing slope;
see \cref{section:proof} for the precise definitions.
We stress that this is one inequality for each non-trivial Harder--Narasimhan type~$\tuple{d}^*$;
for the results of \cite{MR4954467} to hold,
we would like \cref{equation:rigidity-inequality} to be satisfied for all of them.

\vspace{-0.5\baselineskip}
\paragraph{Two implications}
Between these notions, the following two implications hold.
\begin{equation}
  \label{equation:implications}
  \text{$\tuple{d}$ is strongly amply $\theta$\nobreakdash-stable}
  \Rightarrow
  \text{\eqref{equation:rigidity-inequality} holds for \emph{every} non-trivial $\tuple{d}^*$}
  \Rightarrow
  \text{$\tuple{d}$ is amply $\theta$\nobreakdash-stable}.
\end{equation}
The first implication is \cite[Proposition~4.9]{MR4954467},
and the second is \cite[Lemma~4.10]{MR4954467}, which assumes~$\tuple{d}$ to be~$\theta$\nobreakdash-coprime.\footnote{%
  Coprimality enters here because \eqref{equation:rigidity-inequality} bounds the codimensions of the non-trivial Harder--Narasimhan strata,
  which cover the complement of the \emph{semistable} locus,
  whereas ample stability (\cref{definition:amply-stable}) concerns the complement of the \emph{stable} locus;
  these coincide precisely when every~$\theta$\nobreakdash-semistable representation of dimension vector~$\tuple{d}$ is~$\theta$\nobreakdash-stable,
  for which~$\theta$\nobreakdash-coprimality is the standard sufficient condition.
  The converse, which is \cref{theorem:main}, needs no such assumption.
}
By \cite[Examples~4.8 and~4.14]{MR4954467} the first implication is strict.
The main result of this note is the converse of the second implication,
thus answering the question raised in \cite[Remark~1.3]{MR4954467}.
\begin{alphatheorem}
  \label{theorem:main}
  Let~$Q$ be a quiver.
  If $\tuple{d}$ is amply~$\theta$\nobreakdash-stable,
  then the rigidity inequality \eqref{equation:rigidity-inequality} holds
  for every non-trivial Harder--Narasimhan type~$\tuple{d}^*$.
\end{alphatheorem}

This makes it possible to get rid of a technical condition in the following recent results on quiver moduli:
\begin{itemize}[noitemsep]
  \item the higher cohomology vanishing
    in \cite[Theorem~A]{MR4954467};
  \item Schofield's conjecture that the universal representation is partial tilting
    in \cite[Corollary~C]{MR4954467};
  \item the rigidity for quiver moduli for acyclic quivers
    in \cite[Corollary~D]{MR4954467};
  \item the isomorphism between vector fields and the first Hochschild cohomology of~$\field Q$ for acyclic quivers
    in \cite[Theorem~B]{MR4883104};
  \item the fully faithfulness of the Fourier--Mukai functor attached to the universal representation for acyclic quivers
    in \cite[Theorem~D]{MR4883104}.
\end{itemize}
All these results were phrased in terms of strong ample stability.
However, as explained in \cite[Remark~1.2]{MR4954467},
for the proofs in \cite{MR4954467,MR4883104}
it actually suffices to assume \eqref{equation:rigidity-inequality}
for every non-trivial Harder--Narasimhan type.
By \cref{theorem:main}
it now suffices to simply check ample stability
(e.g., using \cite{quivertools}),
which is a natural condition with a clear geometric interpretation.

Combined with the second implication of \eqref{equation:implications},
\cref{theorem:main} shows that,
whenever~$\theta$\nobreakdash-semistability and~$\theta$\nobreakdash-stability coincide for~$\tuple{d}$,
ample stability is equivalent to the rigidity inequality
(\cref{corollary:equivalence}).
Thus, in the~$\theta$\nobreakdash-coprime setting of \cite{MR4954467,MR4883104},
ample stability is precisely the condition needed for the Teleman argument of
\cite[Theorem~A]{MR4954467}.
It is not, however, necessary for the resulting vanishing,
as shown by \cite[Example~4.15]{MR4954467}.

\vspace{-0.5\baselineskip}
\paragraph{Idea of the proof}
We will consider the first piece~$\tuple{d}^1$
of non-trivial Harder--Narasimhan types.
Semistability implies that~$-\langle\tuple{d}^1,\tuple{d}-\tuple{d}^1\rangle\geq 1$
by \cref{lemma:positivity},
and the rigidity inequality holds whenever this bound is strict
by \cref{lemma:rewriting}.
The proof then continues with the observation that equality can be used to produce
a \emph{possibly different} Harder--Narasimhan stratum of codimension~$1$,
by replacing the remaining pieces with the generic Harder--Narasimhan type
of~$\tuple{d}-\tuple{d}^1$, leading to \cref{proposition:not-amply-stable}.
Since ample stability rules out such a stratum,
the rigidity inequality follows.

\vspace{-0.5\baselineskip}
\paragraph{Acknowledgements}
P.B.~was partially supported by NWO (\href{https://doi.org/10.61686/RZKLF82806}{\texttt{doi:10.61686/RZKLF82806}}).
G.P.~was supported by the Luxembourg National Research Fund (FNR--17953441)
and by the Korteweg--de Vries Instituut, at the University of Amsterdam.
The strategy of using the Harder--Narasimhan type~$\tuple{e}_\gen^*$ in \cref{proposition:not-amply-stable}
was found using an LLM.

\section{Proof}
\label{section:proof}
Throughout, we let~$Q$ denote a quiver\footnote{
  Whilst \cite{MR1974891} assumes acyclic throughout,
  this is not necessary for the Harder--Narasimhan formalism
  nor for the codimension formula.
}, with vertices $Q_{0}$ and arrows $Q_{1}$.
Let~$\tuple{d}\in\mathbb{N}^{Q_0}$ be a dimension vector,
and~$\theta\in\Hom(\mathbb{Z}^{Q_0},\mathbb{Z})$ be a stability parameter in the sense of King \cite{MR1315461},
with~$\theta(\tuple{d})=0$.
We write~$\lvert\tuple{e}\rvert\colonequals\sum_{i\in Q_0}e_i$ for the total dimension of a dimension vector~$\tuple{e}$,
and~$\mu(\tuple{e})\colonequals\theta(\tuple{e})/\lvert\tuple{e}\rvert$ for its slope.
We recall the definition of the \emph{Euler form} on~$\mathbb{Z}^{Q_0}$,
\begin{equation}
  \langle\tuple{d},\tuple{e}\rangle\colonequals\sum_{i\in Q_0}d_ie_i-\sum_{\alpha\in Q_1}d_{\mathrm{s}(\alpha)}e_{\mathrm{t}(\alpha)};
\end{equation}
it satisfies~$\langle\tuple{d},\tuple{e}\rangle=\dim\Hom(M,N)-\dim\Ext^1(M,N)$
for representations~$M$ and~$N$ of dimension vectors~$\tuple{d}$ and~$\tuple{e}$.
For more background on quiver moduli and the Harder--Narasimhan stratification
we refer to \cite{MR1974891} and to the survey \cite{MR2484736}.

Let us recall that a representation~$M$ is~\emph{$\mu$-semistable}, resp.~\emph{$\mu$-stable},
if~$\mu(\dimvect N)\leq\mu(\dimvect M)$,
resp.~$\mu(\dimvect N)<\mu(\dimvect M)$,
for every proper non-zero subrepresentation~$N\subseteq M$.
When~$\theta(\dimvect M)=0$,
as is the case for representations of dimension vector~$\tuple{d}$,
this is equivalent to~$\theta(\dimvect N)\leq 0$ (resp.~$\theta(\dimvect N)<0$)
for all such~$N$,
and we then say that~$M$ is~\emph{$\theta$\nobreakdash-(semi)stable}.
Likewise,
a dimension vector~$\tuple{e}$ is~\emph{$\mu$-(semi)stable}
(or~\emph{$\theta$\nobreakdash-(semi)stable} when~$\theta(\tuple{e})=0$)
if there exists a~$\mu$-(semi)stable representation
of dimension vector~$\tuple{e}$,
which translates to the non-emptiness of the corresponding moduli spaces.
Representations of dimension vector~$\tuple{e}$ are parametrised by the affine space~$\repspace{Q,\tuple{e}}$,
in which the~$\mu$-semistable, resp.~$\mu$-stable, representations form open loci;
we denote the~$\theta$\nobreakdash-stable locus by~$\repspace[\theta\stable]{Q,\tuple{d}}$.

The dimension vector~$\tuple{d}$ is~\emph{$\theta$\nobreakdash-coprime}
if~$\theta(\tuple{e})\neq\theta(\tuple{d})$ for every proper non-zero subdimension vector~$\tuple{e}$ of~$\tuple{d}$;
in this case every~$\theta$\nobreakdash-semistable representation of dimension vector~$\tuple{d}$ is~$\theta$\nobreakdash-stable,
so that the stable and semistable loci of~$\repspace{Q,\tuple{d}}$ coincide.
In fact, the loci coincide if and only if
there is no decomposition~$\tuple{d}=\tuple{e}+\tuple{f}$
into non-zero~$\theta$\nobreakdash-semistable dimension vectors with~$\theta(\tuple{e})=0$,
which can be checked using \cite{quivertools}.

We recall the following from \cite{MR1974891}.
\begin{definition}
  \label{definition:hn-type}
  Let~$\tuple{d}$ be a dimension vector.
  A \emph{Harder--Narasimhan type} of~$\tuple{d}$ is
  a tuple~$\tuple{d}^*=(\tuple{d}^1,\ldots,\tuple{d}^\ell)$
  of non-zero~$\mu$-semistable dimension vectors such that~$\sum_{r=1}^\ell\tuple{d}^r=\tuple{d}$,
  and~$\mu(\tuple{d}^1)>\ldots>\mu(\tuple{d}^\ell)$.
  We say that~$\ell$ is the \emph{length},
  and that~$\tuple{d}^*$
  is \emph{non-trivial} if~$\ell\geq 2$.
\end{definition}
Every representation admits a unique Harder--Narasimhan filtration,
and op.~cit. stratifies~$\repspace{Q,\tuple{d}}$
by attaching the \emph{Harder--Narasimhan stratum}
\begin{equation}
  \label{equation:hn-stratum}
  \stratum_{\tuple{d}^*}
  \colonequals
  \{M\in\repspace{Q,\tuple{d}}\mid\text{$M$ has Harder--Narasimhan type $\tuple{d}^*$}\}
  \subseteq\repspace{Q,\tuple{d}}
\end{equation}
to the Harder--Narasimhan type~$\tuple{d}^*$ \cite[Definition~3.3]{MR1974891}.
By~\cite[Proposition~3.4]{MR1974891} we have that
\begin{equation}
  \label{equation:codimension}
  \codim_{\repspace{Q,\tuple{d}}}\stratum_{\tuple{d}^*}
  =
  -\sum_{1\leq m<n\leq\ell}\langle\tuple{d}^m,\tuple{d}^n\rangle.
\end{equation}
This stratification plays an important role in
applying Teleman quantisation to quiver moduli \cite{MR4954467}.

We also recall the following notion from \cite[Definition~4.1]{MR3683503}.
\begin{definition}
  \label{definition:amply-stable}
  The dimension vector~$\tuple{d}$ is \emph{amply $\theta$\nobreakdash-stable}
  if the non-stable locus~$\repspace{Q,\tuple{d}}\setminus\repspace[\theta\stable]{Q,\tuple{d}}$
  has codimension at least~2 in~$\repspace{Q,\tuple{d}}$.
\end{definition}
Every non-trivial Harder--Narasimhan stratum is contained in the non-stable locus,
so if~$\tuple{d}$ is amply~$\theta$\nobreakdash-stable
then~$\codim_{\repspace{Q,\tuple{d}}}\stratum_{\tuple{d}^*}\geq 2$
for every non-trivial Harder--Narasimhan type~$\tuple{d}^*$.
This codimension condition implies for instance
that~$\Br(\modulispace[\theta\stable]{Q,\tuple{d}})\cong\mathbb{Z}/\gcd(\tuple{d})\mathbb{Z}$ \cite[Theorem~4.2]{MR3683503},
and,
if~$\tuple{d}$ is moreover~$\theta$\nobreakdash-coprime,
that~$\Pic(\modulispace[\theta\stable]{Q,\tuple{d}})\cong\mathbb{Z}^{\#Q_0-1}$
\cite[Proposition~3.1]{MR4352662}.

Finally,
we recall the following sufficient and easy-to-check condition
for ample stability \cite[Proposition~5.1]{MR3683503}.
\begin{definition}
  \label{definition:strongly-amply-stable}
  The dimension vector~$\tuple{d}$ is \emph{strongly amply $\theta$\nobreakdash-stable}
  if for every destabilising proper non-zero subdimension vector~$\tuple{e}$
  (i.e., $\theta(\tuple{e})>0$)
  we have that~$-\langle\tuple{e},\tuple{d}-\tuple{e}\rangle\geq2$.
\end{definition}
This notion is studied in depth in \cite{martinez-acosta-thesis}.

\subsection{Properties of Harder--Narasimhan types}
\label{subsection:hn-types}
We will now introduce notation and terminology,
some of which already appeared implicitly in \cite[\S4.2]{MR4954467}.
We will use the notion of \emph{general ext} from \cite{MR1162487}:
$\hom(\tuple{e},\tuple{f})$, resp.~$\ext(\tuple{e},\tuple{f})$,
denotes the value of~$\dim\Hom(M,N)$, resp.~$\dim\Ext^1(M,N)$,
for general representations~$M$ and~$N$ of dimension vectors~$\tuple{e}$ and~$\tuple{f}$;
by semicontinuity this general value is the minimal one, see \cite[p.~48]{MR1162487}.
\vspace{-0.5\baselineskip}
\paragraph{Notation}
To a Harder--Narasimhan type~$\tuple{d}^*$
we associate the \emph{general ext numbers}
\begin{equation}
  \mathrm{E}_{m,n}\colonequals -\langle\tuple{d}^m,\tuple{d}^n\rangle
  \qquad\text{for }1\leq m<n\leq\ell.
\end{equation}
The general representations of dimension vectors $\tuple{d}^{m}$ and $\tuple{d}^{n}$
are again $\mu$-semistable,
so by the slope inequality~$\hom(\tuple{d}^m,\tuple{d}^n)=0$
and hence~$\mathrm{E}_{m,n}=\ext(\tuple{d}^m,\tuple{d}^n)\geq 0$,
explaining the terminology.
We also introduce the partial sums~$\tuple{a}_r\colonequals\sum_{m=1}^r\tuple{d}^m$
for~$r=1,\ldots,\ell-1$,
and define the \emph{cut numbers}
\begin{equation}
  \label{equation:cut-numbers}
  \mathrm{N}_r\colonequals-\langle\tuple{a}_r,\tuple{d}-\tuple{a}_r\rangle
  =\sum_{m=1}^r\sum_{n=r+1}^\ell\mathrm{E}_{m,n}
\end{equation}
(the second equality holds by bilinearity of the Euler form),
and the \emph{slope gaps}
\begin{equation}
  \label{equation:slope-gaps}
  \delta_r\colonequals\mu(\tuple{d}^r)-\mu(\tuple{d}^{r+1}).
\end{equation}

\vspace{-0.5\baselineskip}
\paragraph{The Harder--Narasimhan polygon}
These numbers are conveniently visualised
using the \emph{Harder--Narasimhan polygon} of~$\tuple{d}^*$,
pictured in \cref{figure:polygon,figure:polygon-example}.
Setting~$\tuple{a}_0\colonequals0$ and~$\tuple{a}_\ell\colonequals\tuple{d}$,
this polygon is the piecewise-linear path in the plane
joining the points~$(\lvert\tuple{a}_r\rvert,\theta(\tuple{a}_r))$ for~$r=0,\ldots,\ell$.
As~$\theta(\tuple{d})=0$, both endpoints lie on the horizontal axis.
Its~$r$th edge has slope~$\mu(\tuple{d}^r)$,
so the condition~$\mu(\tuple{d}^1)>\ldots>\mu(\tuple{d}^\ell)$ in \cref{definition:hn-type}
says precisely that the polygon is concave.

In what follows, the \emph{polygon of a representation}~$M$
is the polygon of its Harder--Narasimhan type.

The polygon originates in the theory of vector bundles,
where Shatz proved that,
in an algebraic family of vector bundles on a projective variety,
the Harder--Narasimhan polygon rises under specialisation \cite[Theorem~3]{MR0498573}.
Shatz calls the polygon \emph{convex},
viewing it as the upper boundary of the convex hull of the points of all subsheaves;
as a graph it is concave, which is the convention we use.
For quiver representations,
the polygon and the pointwise partial order comparing polygons
were introduced in \cite[Definition~3.6]{MR1974891}.

\begin{figure}[t]
  \centering
  \begin{tikzpicture}[x=1.05cm,y=0.8cm,>=stealth]
    \coordinate (P0) at (0,0);
    \coordinate (P1) at (1.5,3);
    \coordinate (P2) at (4,4.5);
    \coordinate (P3) at (6,4);
    \coordinate (P4) at (8,0);
    \draw[dashed] (P0) -- (P4);
    \node[below] at (4,0) {slope $\mu(\tuple{d})=0$};
    \draw[gray] (P0) -- (P2);
    \draw[gray] (P2) -- (P4);
    \node[gray] at (2.9,2.2) {$\mu(\tuple{a}_2)$};
    \node[gray] at (5.1,2.0) {$\mu(\tuple{d}-\tuple{a}_2)$};
    \draw[very thick] (P0) -- (P1) -- (P2) -- (P3) -- (P4);
    \node at (0.35,2.1) {$\mu(\tuple{d}^1)$};
    \node at (2.3,4.15) {$\mu(\tuple{d}^2)$};
    \node at (5.4,4.7) {$\mu(\tuple{d}^3)$};
    \node at (7.7,2.2) {$\mu(\tuple{d}^4)$};
    \foreach \p in {P0,P1,P2,P3,P4} {\fill (\p) circle (1.7pt);}
    \node[below left] at (P0) {$0$};
    \node[left] at (P1) {$\tuple{a}_1=\tuple{d}^1$};
    \node[above] at (P2) {$\tuple{a}_2$};
    \node[above right] at (P3) {$\tuple{a}_3$};
    \node[right] at (P4) {$\tuple{d}$};
  \end{tikzpicture}
  \caption{The Harder--Narasimhan polygon of a type~$\tuple{d}^*$ of length~$\ell=4$.
    As~$\theta(\tuple{d})=0$,
    the dashed segment joining the endpoints is horizontal,
    and every interior vertex~$\tuple{a}_r$ lies strictly above it,
    so that~$\mu(\tuple{a}_r)>\mu(\tuple{d})>\mu(\tuple{d}-\tuple{a}_r)$.
  }
  \label{figure:polygon}
\end{figure}
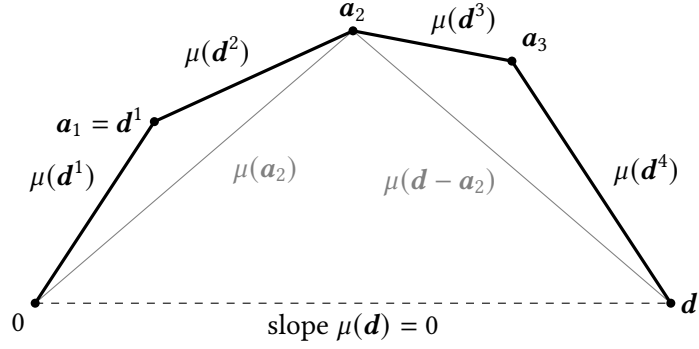

\begin{figure}[t]
  \centering
  \begin{tikzpicture}[x=0.95cm,y=1.1cm,>=stealth]
    \coordinate (P0) at (0,0);
    \coordinate (P1) at (2,2);
    \coordinate (P2) at (5,3);
    \coordinate (P3) at (9,1);
    \coordinate (P4) at (10,0);
    \draw[dashed] (P0) -- (P4);
    \node[below] at (5,0) {slope $\mu(\tuple{d})=0$};
    \draw[gray,thick] (P0) -- (4,2) -- (6,2) -- (P4);
    \foreach \p in {(4,2),(6,2)} {\fill[gray] \p circle (1.7pt);}
    \node[gray,below] at (2.1,0.8) {$\mu(\tuple{e}^1)$};
    \node[gray,below] at (5,1.95) {$\mu(\tuple{e}^2)$};
    \node[gray,below] at (8,0.95) {$\mu(\tuple{e}^3)$};
    \draw[very thick] (P0) -- (P1) -- (P2) -- (P3) -- (P4);
    \node at (0.6,1.5) {$\mu(\tuple{d}^1)$};
    \node at (3.2,2.95) {$\mu(\tuple{d}^2)$};
    \node at (7.3,2.45) {$\mu(\tuple{d}^3)$};
    \node at (10.05,0.8) {$\mu(\tuple{d}^4)$};
    \foreach \p in {P0,P1,P2,P3,P4} {\fill (\p) circle (1.7pt);}
    \node[below left] at (P0) {$0$};
    \node[above left] at (P1) {$\tuple{a}_1=\tuple{d}^1$};
    \node[above] at (P2) {$\tuple{a}_2$};
    \node[above right] at (P3) {$\tuple{a}_3$};
    \node[right] at (P4) {$\tuple{d}$};
  \end{tikzpicture}
  \caption{An explicit instance of \cref{figure:polygon},
    for the $3$-Kronecker quiver with~$\tuple{d}=(5,5)$ and~$\theta=(1,-1)$:
    the Harder--Narasimhan polygon of the type~$\tuple{d}^*=((2,0),(2,1),(1,3),(0,1))$ (black),
    and that of the type~$\tuple{e}^*=((3,1),(1,1),(1,3))$ (grey),
    which lies pointwise below it.
    As~$(5,5)$ is~$\theta$\nobreakdash-semistable,
    the pointwise smallest polygon is the dashed segment,
    the polygon of the trivial type~$((5,5))$.
  }
  \label{figure:polygon-example}
\end{figure}
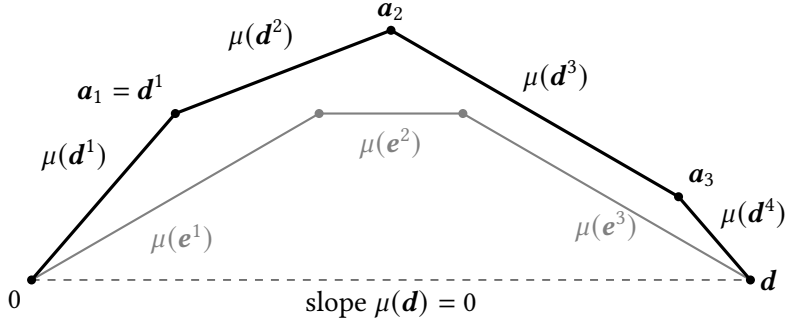

The slope gap~$\delta_r$ measures the decrease in slope of the polygon at the vertex~$\tuple{a}_r$,
whilst the cut number~$\mathrm{N}_r$ measures the general ext across the cut from~$\tuple{a}_r$ to~$\tuple{d}-\tuple{a}_r$.
The right-hand side of the rigidity inequality \eqref{equation:rigidity-inequality}
uses a weighted sum of these numbers, as explained by \cref{lemma:rewriting}.
The two geometric facts we use below are visible in \cref{figure:polygon,figure:polygon-example}:
\begin{itemize}[noitemsep]
  \item each interior vertex~$\tuple{a}_r$ lies strictly above the dashed segment
    (used in \cref{lemma:positivity});
  \item the general representation of dimension vector~$\tuple{d}$
    has the pointwise smallest polygon,
    so that the slope of its first edge is as small as possible
    (used in \cref{proposition:not-amply-stable}).
\end{itemize}
The second fact is \cref{lemma:semicontinuity} below,
the analogue for quiver representations of Shatz's specialisation theorem,
where we compare polygons with the same endpoints using the pointwise partial order.
As this analogue does not seem to be recorded in the quiver literature yet,
we include a proof;
it is a standard translation of Shatz's arguments.

\begin{lemma}
  \label{lemma:semicontinuity}
  For every concave polygon~$P$
  the locus of representations in~$\repspace{Q,\tuple{d}}$
  whose Harder--Narasimhan polygon lies on or above~$P$
  is closed.
  In particular,
  the polygon of the general representation of dimension vector~$\tuple{d}$
  lies below that of every other Harder--Narasimhan type of~$\tuple{d}$.
\end{lemma}

\begin{proof}
  For a tuple~$(\tuple{f}^1,\ldots,\tuple{f}^s)$ of non-zero dimension vectors summing to~$\tuple{d}$,
  the locus of representations
  admitting a filtration with successive subquotients of dimension vectors~$\tuple{f}^1,\ldots,\tuple{f}^s$
  is closed in~$\repspace{Q,\tuple{d}}$:
  it is the image of the universal quiver flag variety,
  which is projective over~$\repspace{Q,\tuple{d}}$
  (cf.~the universal quiver Grassmannian in \cite[\S3]{MR1162487}).
  To such a filtration we associate the piecewise-linear path
  joining the points~$(\sum_{i\leq r}\lvert\tuple{f}^i\rvert,\sum_{i\leq r}\theta(\tuple{f}^i))$
  for~$r=0,\ldots,s$;
  for a Harder--Narasimhan filtration this path is the Harder--Narasimhan polygon.

  Now, we claim that the polygon of a representation~$M$ lies on or above~$P$
  if and only if~$M$ admits a filtration whose path lies on or above~$P$.
  The Harder--Narasimhan filtration gives one implication.
  For the converse, it suffices to show that every subrepresentation~$N\subseteq M$
  defines a point~$(|\dimvect{N}|, \theta(\dimvect{N}))$ on or below the polygon of~$M$;
  this is the \emph{maximal property} of the Harder--Narasimhan polygon,
  cf.~\cite[Theorem~2]{MR0498573}.
  The path of any filtration then lies below the concave polygon,
  and above~$P$ by assumption.
  For quiver representations the maximal property follows
  by intersecting~$N$ with the Harder--Narasimhan filtration of~$M$:
  the subquotients of the induced filtration of~$N$
  embed into the semistable subquotients of~$M$,
  so its path consists of steps of width at most~$\lvert\tuple{d}^r\rvert$
  and slope at most~$\mu(\tuple{d}^r)$,
  and every such path ends on or below the polygon.

  As only finitely many tuples~$(\tuple{f}^1,\ldots,\tuple{f}^s)$ occur,
  the locus in the statement is a finite union of closed subsets, hence closed.
  The second statement follows
  because~$\repspace{Q,\tuple{d}}$ is irreducible:
  the closed locus of representations
  whose polygon lies on or above the polygon of the general representation
  contains the dense Harder--Narasimhan stratum,
  hence it is all of~$\repspace{Q,\tuple{d}}$.
  The polygon of the general representation
  corresponds to the Harder--Narasimhan type
  that gives the dense stratum,
  and it is the only one to lie on itself,
  so every other Harder--Narasimhan polygon lies
  above it.
\end{proof}

The cut numbers~$\mathrm{N}_r$ already appeared in the proof of \cite[Proposition~4.9]{MR4954467}.
We will use the fact that~$\mathrm{N}_r\geq 1$ for all~$r=1,\ldots,\ell-1$,
for which we provide a self-contained proof,
as~loc.~cit.~omits those details.

\begin{lemma}
  \label{lemma:positivity}
  Let~$\tuple{d}^*$ be a non-trivial Harder--Narasimhan type for~$\tuple{d}$.
  If~$\tuple{d}$ is~$\theta$\nobreakdash-semistable,
  then~$\mathrm{N}_r\geq 1$ for~$r=1,\ldots,\ell-1$.
\end{lemma}

\begin{proof}
  Let~$M\in\stratum_{\tuple{d}^*}$,
  with Harder--Narasimhan filtration~$0=M_0\subset\ldots\subset M_\ell=M$,
  and set~$F\colonequals M_r$ and~$G\colonequals M/M_r$,
  of dimension vectors~$\tuple{a}_r$ and~$\tuple{d}-\tuple{a}_r$.
  As~$F$ is built from semistable pieces of slope at least~$\mu(\tuple{d}^r)>\mu(\tuple{d}^{r+1})$,
  and those of~$G$ have slope at most~$\mu(\tuple{d}^{r+1})$,
  we have~$\Hom(F,G)=0$,
  so
  \begin{equation}
    \mathrm{N}_r=-\langle\tuple{a}_r,\tuple{d}-\tuple{a}_r\rangle=\dim\Ext^1(F,G)\geq\ext(\tuple{a}_r,\tuple{d}-\tuple{a}_r).
  \end{equation}
  Moreover~$\mu(\tuple{a}_r)>\mu(\tuple{d})$,
  since~$\mu(\tuple{d})$ is a weighted average of~$\mu(\tuple{a}_r)\geq\mu(\tuple{d}^r)$
  and~$\mu(\tuple{d}-\tuple{a}_r)\leq\mu(\tuple{d}^{r+1})$;
  as~$\mu(\tuple{d})=0$,
  this says that the vertex~$\tuple{a}_r$ lies strictly above the dashed segment
  in \cref{figure:polygon,figure:polygon-example}.
  As~$\tuple{d}$ is~$\theta$\nobreakdash-semistable,
  the general representation of dimension vector~$\tuple{d}$ is~$\theta$\nobreakdash-semistable,
  hence has no subrepresentation of slope larger than~$\mu(\tuple{d})$,
  in particular none of dimension vector~$\tuple{a}_r$.
  As the general representation of dimension vector~$\tuple{d}$
  admits a subrepresentation of dimension vector~$\tuple{a}_r$
  if and only if~$\ext(\tuple{a}_r,\tuple{d}-\tuple{a}_r)=0$
  \cite[Theorems~3.3 and~5.4]{MR1162487},
  this forces~$\ext(\tuple{a}_r,\tuple{d}-\tuple{a}_r)\geq 1$.
  Therefore~$\mathrm{N}_r\geq 1$.
\end{proof}

\begin{lemma}
  \label{lemma:rewriting}
  Let~$\tuple{d}^*$ be a non-trivial Harder--Narasimhan type.
  Then \eqref{equation:rigidity-inequality}
  is equivalent to~$\sum_{r=1}^{\ell-1}\delta_r(\mathrm{N}_r-1)>0$.
\end{lemma}

\begin{proof}
  We have that~$\mu(\tuple{d}^m)-\mu(\tuple{d}^n)=\sum_{r=m}^{n-1}\delta_r$,
  so the right-hand side of \eqref{equation:rigidity-inequality} can be rewritten as
  \begin{equation}
    \label{equation:rewriting}
    \sum_{m<n}\mathrm{E}_{m,n}\sum_{r=m}^{n-1}\delta_r
    =
    \sum_{r=1}^{\ell-1}\delta_r\sum_{m\leq r<n}\mathrm{E}_{m,n}
    =
    \sum_{r=1}^{\ell-1}\delta_r\mathrm{N}_r
  \end{equation}
  by swapping the order of the summation.
  The left-hand side of \eqref{equation:rigidity-inequality} is~$\mu(\tuple{d}^1)-\mu(\tuple{d}^\ell)=\sum_{r=1}^{\ell-1}\delta_r$,
  so \eqref{equation:rigidity-inequality} is equivalent to
  $\sum_{r=1}^{\ell-1}\delta_r\mathrm{N}_r>\sum_{r=1}^{\ell-1}\delta_r$,
  that is, to~$\sum_{r=1}^{\ell-1}\delta_r(\mathrm{N}_r-1)>0$.
\end{proof}

This leads to the following way of proving \cref{theorem:main}.

\begin{proposition}
  \label{proposition:reduction}
  Let~$\tuple{d}^*$ be a non-trivial Harder--Narasimhan type for~$\tuple{d}$.
  If~$\tuple{d}$ is~$\theta$\nobreakdash-semistable,
  then \eqref{equation:rigidity-inequality} holds,
  unless~$\mathrm{N}_r=1$ for all~$r=1,\ldots,\ell-1$.
\end{proposition}

\begin{proof}
  By \cref{lemma:positivity} we have~$\mathrm{N}_r\geq 1$,
  so~$\sum_{r=1}^{\ell-1}\delta_r(\mathrm{N}_r-1)\geq 0$ because~$\delta_r>0$.
  This sum vanishes if and only if~$\mathrm{N}_r=1$ for all~$r$,
  so otherwise \eqref{equation:rigidity-inequality} holds.
\end{proof}

In particular,
note that \eqref{equation:rigidity-inequality} always holds in its non-strict version
for Harder--Narasimhan types of~$\theta$\nobreakdash-semistable dimension vectors.

\subsection{Constructing the special Harder--Narasimhan type}
\label{subsection:construction-all-cuts-one}
The following lemma will provide the crucial step in the proof.
Recall that the slope of a dimension vector is~$\mu(\tuple{e})=\theta(\tuple{e})/\lvert\tuple{e}\rvert$,
and that Harder--Narasimhan types are taken with respect to it.
\begin{lemma}
  \label{lemma:additivity}
  Let~$\tuple{d}^*=(\tuple{d}^1,\ldots,\tuple{d}^\ell)$ be a non-trivial Harder--Narasimhan type for~$\tuple{d}$.
  Write~$\tuple{e}\colonequals\tuple{d}-\tuple{d}^1$,
  and~$\tuple{e}^*\colonequals(\tuple{d}^2,\ldots,\tuple{d}^\ell)$.
  Then~$\tuple{e}^*$ is a Harder--Narasimhan type for~$\tuple{e}$,
  and
  \begin{equation}
    \label{equation:additivity}
    \codim_{\repspace{Q,\tuple{d}}}\stratum_{\tuple{d}^*}
    =
    -\langle\tuple{d}^1,\tuple{d}-\tuple{d}^1\rangle
    +\codim_{\repspace{Q,\tuple{d}-\tuple{d}^1}}\stratum_{\tuple{e}^*}.
  \end{equation}
\end{lemma}

\begin{proof}
  For the first part,
  observe that subsequences of Harder--Narasimhan types are again Harder--Narasimhan types,
  so~$\tuple{e}^*$ is a Harder--Narasimhan type for~$\tuple{e}$,
  with associated stratum~$\stratum_{\tuple{e}^*}$.
  The additivity follows by regrouping the terms in \eqref{equation:codimension}
  and using the bilinearity of the Euler form.
\end{proof}

We can now prove the following result.
We will come back to the fact that we only consider~$\mathrm{N}_1=1$
(and not~$\mathrm{N}_r=1$ for some~$r\in\{1,\ldots,\ell-1\}$)
in \cref{corollary:equivalence,remark:curiosity}.
For the moment
it suffices to observe that
the proposition does not say that~$\stratum_{\tuple{d}^*}$ has codimension~1;
it only asserts the existence of such a stratum.

\begin{proposition}
  \label{proposition:not-amply-stable}
  Let~$\tuple{d}^*=(\tuple{d}^1,\ldots,\tuple{d}^\ell)$
  be a non-trivial Harder--Narasimhan type
  for which~$\mathrm{N}_1=1$.
  Then~$\repspace{Q,\tuple{d}}$ admits a non-trivial Harder--Narasimhan stratum of codimension~1;
  in particular~$\tuple{d}$ is not~amply $\theta$\nobreakdash-stable.
\end{proposition}

\begin{proof}
  Write~$\tuple{e}\colonequals\tuple{d}-\tuple{d}^1$.
  Consider the Harder--Narasimhan stratification of~$\repspace{Q,\tuple{e}}$,
  for the slope stability~$\mu$,
  and let~$\tuple{e}_\gen^*=(\tuple{e}^2,\ldots,\tuple{e}^t)$
  be the generic Harder--Narasimhan type,
  i.e.,
  the type associated to the
  unique dense open
  Harder--Narasimhan stratum of codimension~0
  in~$\repspace{Q,\tuple{e}}$.
  If~$\tuple{e}$ is~$\mu$-semistable,
  then~$\tuple{e}_\gen^*=(\tuple{e})$ and~$t=2$,
  and otherwise~$t\geq 3$;
  the construction below treats both cases uniformly.

  Define~$\tuple{f}^*\colonequals(\tuple{d}^1,\tuple{e}^2,\ldots,\tuple{e}^t)$.
  We claim that~$\tuple{f}^*$ is a non-trivial
  Harder--Narasimhan type for~$\tuple{d}$,
  whose Harder--Narasimhan stratum~$\stratum_{\tuple{f}^*}$ has codimension~1.

  By construction,~$\tuple{d}^1,\tuple{e}^2,\ldots,\tuple{e}^t$
  are non-zero and~$\mu$-semistable,
  and they sum to~$\tuple{d}$.
  We also have that~$\mu(\tuple{e}^2)>\ldots>\mu(\tuple{e}^t)$,
  so it suffices to check that~$\mu(\tuple{d}^1)>\mu(\tuple{e}^2)$.
  Now,
  $(\tuple{d}^2,\ldots,\tuple{d}^\ell)$ is itself a Harder--Narasimhan type for~$\tuple{e}$,
  realised by representations
  such that their subrepresentations have slope at most~$\mu(\tuple{d}^2)$.
  By \cref{lemma:semicontinuity},
  the generic representation of dimension vector~$\tuple{e}$ has the pointwise smallest polygon,
  hence the shallowest first edge and the smallest maximal slope;
  its first piece~$\tuple{e}^2$ thus satisfies~$\mu(\tuple{e}^2)\leq\mu(\tuple{d}^2)<\mu(\tuple{d}^1)$.

  The claim on the codimension follows for this choice,
  as by \cref{lemma:additivity}
  \begin{equation}
    \label{equation:codimension-1}
    \begin{aligned}
      \codim_{\repspace{Q,\tuple{d}}}\stratum_{\tuple{f}^*}
      &=-\langle\tuple{d}^1,\tuple{d}-\tuple{d}^1\rangle+\codim_{\repspace{Q,\tuple{d}-\tuple{d}^1}}\stratum_{\tuple{e}_\gen^*} \\
      &=\mathrm{N}_1+0=1,
    \end{aligned}
  \end{equation}
  using that~$-\langle\tuple{d}^1,\tuple{d}-\tuple{d}^1\rangle=\mathrm{N}_1=1$ by assumption,
  and that the general Harder--Narasimhan type~$\tuple{e}_\gen^*$ for~$\tuple{d}-\tuple{d}^1$ has codimension~0.
  The stratum~$\stratum_{\tuple{f}^*}$
  thus shows that~$\tuple{d}$ is not amply~$\theta$\nobreakdash-stable.
\end{proof}

We record the following observation.

\begin{lemma}
  \label{lemma:codim-1-means-all-cuts-1}
  Let~$\tuple{d}$ be a~$\theta$\nobreakdash-semistable dimension vector.
  Let~$\tuple{d}^*$ be a non-trivial Harder--Narasimhan type for~$\tuple{d}$ of length~$\ell$.
  If~$\codim\stratum_{\tuple{d}^*}=1$,
  then~$\mathrm{N}_r=1$ for~$r=1,\ldots,\ell-1$.
\end{lemma}

\begin{proof}
  By \eqref{equation:codimension} we have~$\sum_{m<n}\mathrm{E}_{m,n}=1$ with each~$\mathrm{E}_{m,n}\geq 0$,
  so a single entry~$\mathrm{E}_{m_0,n_0}$ equals~1 and the others vanish.
  Then~$\mathrm{N}_r=\sum_{m\leq r<n}\mathrm{E}_{m,n}$ equals~1 if~$m_0\leq r<n_0$, and~0 otherwise.
  But~$\mathrm{N}_r\geq 1$ for all~$r$ by \cref{lemma:positivity},
  which forces~$m_0=1$ and~$n_0=\ell$,
  hence~$\mathrm{N}_r=1$ for~$r=1,\ldots,\ell-1$.
\end{proof}

\subsection{Proof of \texorpdfstring{\cref{theorem:main}}{Theorem \ref{theorem:main}}}
\label{subsection:proof}
We can now come to the proof of \cref{theorem:main}.

\begin{proof}[Proof of \cref{theorem:main}]
  Let~$\tuple{d}$ be amply~$\theta$\nobreakdash-stable;
  in particular~$\tuple{d}$ is~$\theta$\nobreakdash-semistable.
  Let~$\tuple{d}^*$ be a non-trivial Harder--Narasimhan type of length~$\ell$.
  By \cref{proposition:reduction}
  the rigidity inequality \eqref{equation:rigidity-inequality} holds,
  unless~$\mathrm{N}_r=1$ for all~$r=1,\ldots,\ell-1$.
  In the latter case~$\mathrm{N}_1=1$,
  so by \cref{proposition:not-amply-stable}
  the dimension vector~$\tuple{d}$ is not amply~$\theta$\nobreakdash-stable,
  a contradiction.
\end{proof}

Let us now state the following corollary,
which makes precise how the first cut number controls ample stability.
\begin{corollary}
  \label{corollary:equivalence}
  Let~$\tuple{d}$ be a dimension vector
  for which the~$\theta$\nobreakdash-semistable and~$\theta$\nobreakdash-stable loci of~$\repspace{Q,\tuple{d}}$ coincide
  (e.g., a~$\theta$\nobreakdash-coprime dimension vector).
  The following are equivalent:
  \begin{enumerate}
    \item $\tuple{d}$ is amply~$\theta$\nobreakdash-stable;
    \item for every non-trivial Harder--Narasimhan type we have that~$\mathrm{N}_1\geq 2$;
    \item for every non-trivial Harder--Narasimhan type the inequality~\eqref{equation:rigidity-inequality} holds.
  \end{enumerate}
\end{corollary}

\begin{proof}
  Assume that~$\tuple{d}$ is amply~$\theta$\nobreakdash-stable.
  If there exists a non-trivial Harder--Narasimhan type for which~$\mathrm{N}_1=1$,
  then by \cref{proposition:not-amply-stable} we get a contradiction
  (and~$\mathrm{N}_1=0$ is excluded by \cref{lemma:positivity},
  as~$\tuple{d}$ is in particular~$\theta$\nobreakdash-semistable).

  Assume that for every non-trivial Harder--Narasimhan type we have that~$\mathrm{N}_1\geq 2$.
  Then~$\tuple{d}$ is~$\theta$\nobreakdash-semistable:
  otherwise the generic Harder--Narasimhan type is non-trivial,
  and all its general ext numbers vanish by \eqref{equation:codimension},
  so that its first cut number is~0.
  Now \eqref{equation:rigidity-inequality} follows from \cref{lemma:rewriting,lemma:positivity}.

  For the final implication we argue by the contrapositive.
  This is \cite[Lemma~4.10]{MR4954467},
  where~$\theta$\nobreakdash-coprimality is a standing assumption,
  only used to ensure that the semistable and stable loci coincide.
  Since our hypothesis is weaker, we include the short argument.
  Assume that~$\tuple{d}$ is not amply $\theta$\nobreakdash-stable.
  If~$\tuple{d}$ is not~$\theta$\nobreakdash-semistable,
  then the generic Harder--Narasimhan type is non-trivial,
  and all its general ext numbers vanish by \eqref{equation:codimension},
  so the right-hand side of \eqref{equation:rigidity-inequality} vanishes
  and the inequality fails.
  If~$\tuple{d}$ is~$\theta$\nobreakdash-semistable,
  then every non-trivial Harder--Narasimhan stratum has positive codimension.
  By assumption, the complement of the stable locus is the non-semistable locus,
  so there exists a non-trivial Harder--Narasimhan type~$\tuple{d}^*$
  whose stratum~$\stratum_{\tuple{d}^*}$ has codimension~1.
  But then \eqref{equation:rigidity-inequality} does not hold by \cref{lemma:codim-1-means-all-cuts-1}
  and \cref{lemma:rewriting}.
\end{proof}

\begin{remark}
  \label{remark:curiosity}
  Note that \cref{lemma:positivity,lemma:rewriting} show that,
  for a fixed non-trivial Harder--Narasimhan type~$\tuple{d}^*$,
  we have that~\eqref{equation:rigidity-inequality} holds
  if and only if~$\mathrm{N}_r\geq 2$ for \emph{some}~$r\in\{1,\ldots,\ell-1\}$.
  The weaker condition~$\mathrm{N}_1=1$ need not make \eqref{equation:rigidity-inequality} fail for~$\tuple{d}^*$ itself;
  nevertheless, by \cref{proposition:not-amply-stable},
  a single type with~$\mathrm{N}_1=1$ already forces~$\tuple{d}$ to fail ample stability,
  by producing a (possibly different) codimension-one type that does violate \eqref{equation:rigidity-inequality}.

  Also,
  note that strong ample stability asks that~$-\langle\tuple{e},\tuple{d}-\tuple{e}\rangle\geq 2$
  for \emph{every} destabilising proper non-zero subdimension vector~$\tuple{e}$,
  i.e.,~one for which~$\mu(\tuple{e})>\mu(\tuple{d})$.
  By \cref{corollary:equivalence},
  ample stability merely asks this for those subdimension vectors~$\tuple{e}$
  which occur as the first entry of a Harder--Narasimhan type.
\end{remark}

\printbibliography

\medskip

\begingroup\small
\emph{Pieter Belmans}, \url{p.belmans@uu.nl} \\
Mathematical Institute, Utrecht University, Budapestlaan 6, 3584CD Utrecht, The Netherlands

\smallskip

\emph{Gianni Petrella}, \url{g.petrella@uva.nl} \\
Korteweg-de Vries Instituut, University of Amsterdam, Science Park 107, 1098XG Amsterdam, The Netherlands
\endgroup

\end{document}